\documentclass[12pt]{amsart}
\usepackage{enumerate}
\usepackage{amsthm,amscd,amssymb,verbatim,epsf,amsmath,amsfonts,mathrsfs,graphicx}
\usepackage[linktocpage,colorlinks=true,linkcolor=blue,citecolor=blue]{hyperref}
\usepackage{lineno}
\newtheorem{thm}{Theorem}
\newtheorem{lem}[thm]{Lemma}

\newtheorem{rem}[thm]{Remark}

\newcommand{\bb}[1]{\mathbb{#1}}

\newcommand{\ip}[2]{\left \langle #1 , #2 \right\rangle}
\newcommand{\n}{\nabla}
\newcommand{\ov}[1]{{\overline{#1}}}
\newcommand{\e}{\epsilon}
\renewcommand{\a}{\alpha}
\renewcommand{\b}{\beta}
\renewcommand{\l}{\lambda}
\newcommand{\g}{\gamma}
\renewcommand{\d}{\delta}
\newcommand{\tr}{\operatorname{tr}}
\newcommand{\on}[1]{{\operatorname{#1}}}

\title[Weyl Estimates for spacelike surfaces]{Weyl Estimates for spacelike hypersurfaces in de Sitter space}
 
\author[D. Ballesteros-Chavez]{Daniel Ballesteros-Chavez}
\address{Daniel Ballesteros-Chavez, Dep. Applied Mathematics \\Silesian University of Technology\\
           Gliwice, 44-100\\
          Poland.}
    \email{daniel.ballesteros-chavez@polsl.pl}
\author[W. Klingenberg]{Wilhelm Klingenberg}
\address{Wilhelm Klingenberg, Department of Mathematics\\ Durham University, Durham, DH1 3LE
           UK.}
    \email{wilhelm.klingenberg@dur.ac.uk}
\author[B. Lambert]{Ben Lambert}
\address{Ben Lambert,
           Department of Electronics, Computing and Mathematics,
    University of Derby, Derby DE22 1GB
          \\UK.}
    \email{b.lambert@derby.ac.uk}
\begin{document}
\maketitle
 
\tableofcontents
\begin{abstract}
    We study the isometric spacelike embedding problem in scaled de Sitter space, and obtain Weyl-type estimates and the corresponding closedness in the space of embeddings.
\end{abstract}
\section{Introduction}
\noindent
The problem of isometric embedding of a closed  surface of genus zero and positive curvature $(S^2,g)$ into Euclidean 3-space
$\mathbb{R}^3$ subsequently known as the Weyl problem was considered by Hermann Weyl in 1916. He  outlined a method to prove the existence of an embedding using the continuity method and proved a curvature estimate. In his landmark thesis, Louis Nirenberg \cite{nirenberg:53}
solved the embedding problem under the assumption that the coefficients of $g$ have continuous fourth order derivatives. More recently, interest in isometric embeddings into semi-Riemannian spaces has been kindled by Mu-Tao Wang and Shing-Tung Yau \cite{Wang-Yau} due to their relation to the Brown-York quasilocal mass in general relativity.

\noindent
A range of related Weyl estimates have previously been investigated in Euclidean spaces. We particularly note that Yanyan Li and Gilbert Weinstein \cite{Li-Weinstein} generalised the Weyl estimates to hypersurfaces in $\mathbb{R}^{n+1}$, and Yanyan Li and Pengfei Guan investigated the degenerate curvature case \cite{GuanLi}. Motivated by \cite{Wang-Yau}, there has been surge of interest in isometric embeddings into Euclidean spaces, as this is a first step in the embedding into Minkowski space. For example , Chen-Yun Lin and Ye-Kai Wang used estimates on isometric embeddings in hyperbolic space to demonstrate the existence of isometric embeddings in Anti de Sitter space \cite{LinWang}. Recently Pengfei Guan and Siyuan Lu~\cite{PGuan-SLu} established an estimate of the mean curvature  of an immersed hypersurface with nonnegative extrinsic scalar curvature in a Riemannian manifold. Then  Chunhe Li and Zhizhang Wang~\cite{Li-Wang} obtained the openness of the collection of embeddable metrics. These results together with the a priori estimates obtained by Siyuan Lu~\cite{luS1} imply existence results for embedding into Riemannian manifolds. 

\noindent
In this paper we study a spacelike isometric embedding
of a metric on the sphere into de Sitter space  $X:(S^n,g) \to (S^{n,1}_1,\bar{g})$, the former being a metric on the sphere with some curvature conditions, and the latter being de Sitter space, that is, the unit sphere in ${\mathbb R}^{n+1,1}$ with the induced metric. Thus, in respect of coordinates on $S^n$, we shall study solutions of the following system of equations for $X$ :
\begin{equation}
  \label{eq:1}
  \bar{g}(\nabla_iX,\nabla_j X) =  g_{ij}.
\end{equation}
Our results are as follows.
\begin{thm}\label{thm:CurvEsts}
Suppose that $g$ is a smooth metric of positive curvature on $S^2$. Then there exists a $\rho_0 > 0$ such that for $0 < \rho < \rho_0$, there is a bound $ C(g, \rho) > 0 $ such that any admissible isometric embedding of $(S^2, g)$ into $(S^{2,1}_{\rho},\overline g)$ satisfies
\begin{equation}
    |A|^2  \leq C(g, \rho).
\end{equation}
Furthermore, if g is a smooth metric on $S^n$ satisfies 
\begin{equation}
\rho^{-2}n(n-2) <  R < \rho^{-2}n(n-1),
\end{equation}
for some
$\rho >0$ and the scalar curvature $R$,
then for any admissible isometric embedding of $(S^n, g)$ into $(S^{n,1}_{\rho},\overline g)$ we have  
    \begin{equation}
    |A|^2 \leq C(n,g,\rho)  .
    \end{equation}
\end{thm}
\noindent
Our second result is a compactness theorem for the graphical representation as in section 2.2, which implies closedness in the space of embeddings.
\begin{thm}
Let  $X:(S^n,g) \to (S^{n,1}_\rho,\bar{g})$ be as in Theorem 1. Then there exists an rotation $R$ of $S^{n,1}_\rho$ such that $RX$ is smooth graph over the round sphere $S^n_\rho \subset S^{n,1}_\rho $. Indeed for all $k\in \mathbb{N}$ there exist constants $C_{k}=C_{k}(n,\rho, g)$ such that the graph function $u$ of $RX$ satisfies 
\[\|u\|_{C^{k}}\leq C_{k}\ .\]
\end{thm}
\begin{rem}
We note that the scaling of de Sitter space is necessary. For example a short argument implies that there can be no isometric embedding of small spheres into de Sitter space of radius 1, see Lemma \ref{lem:Nonexist}.
\end{rem}
\begin{rem}
The isometric embeddings of Mu-Tau Wang and Shing-Tung Yau in \cite{Wang-Yau} are typically codimension 2 embeddings, utilising isometric embedding theorems into Euclidean spaces. The key difference in our approach is that we are considering a codimension 1 embedding into $S^{n,1}$. However, we also note that we could also consider this as a codimension 2 embedding into $\bb{R}^{n+1,1}$.
\end{rem}
\noindent

\section{Geometry of hypersurfaces in de Sitter space}
 
\noindent
We take $\bb{R}^{n+1,1}$ to be $(n+2)$-dimensional Minkowski space, where we choose the standard orthogonal basis $E_0, E_1, \ldots, E_{n+1}$ and $E_0$ has Minkowski norm \[\ip{E_0}{E_0}=:|E_0|^2=-1\ .\] 
We consider the $(n+1)$-dimensional de Sitter space $S^{n,1}_\rho\subset \bb{R}^{n+1, 1}$, which is the Lorentzian manifold consisting of the points $x = (x_0,\ldots,x_{n+1})\in \bb{R}^{n+2}_{1}$ with Minkowski length $\rho$ :
     \begin{equation}
\label{eq:2}
    (S^{n,1}_\rho, \overline g)   :=  \{x \in \bb{R}^{n+2}_{1} | \langle x, x \rangle = \rho^2\} ,
     \end{equation}
     where $\ov{g}$ is the induced metric. Note that here and throughout the paper the index indicates the signature of the metric, so $\bb{R}^{n+1,1}$ is $(n+2)$ dimensional.
     We will write geometric quantities on de Sitter space with a bar, namely $\ov{g}$, $\ov{\n}$, $\ov{A}$, etc. are the induced metric, covariant derivative and second fundamental form. We will write geometric quantities without a bar for quantities on  the embedded hypersurface $(S^n,g) \hookrightarrow (S^{n,1}_\rho, \overline g)$, so for example $g$ and $\n$ for the induced metric and covariant derivative on the isometric immersion $M$, while we will use $\nu$ for the downward pointing timelike unit normal. Finally, we use a double bar for quantities on $\bb{R}^{n+1,1}$, e.g. we will write $\ov{\ov{\n}}$ for the covariant derivative. When indices are repeated we will employ Einstein summation convention where we assume latin letters run in $1\leq i,j,k,\ldots\leq n$ and greek letters run in $0\leq \alpha, \beta, \gamma, \ldots \leq n+1$.
     
     \subsection{Curvature identities} Here we make explicit our chosen sign conventions and collect various standard geometric identities which will be useful later.
     
     For any connection $\n$, we pick the conventions
\[R(X,Y)Z = \n_Y \n_X Z-\n_X \n_Y Z+\n_{[Y,X]} Z, \qquad R(X,Y,Z,A)=\ip{R(X,Y)Z}{A}\ .\]
and for a nondegenerate hypersurface $\widehat{M}^n\subset \widetilde{M}^{n+1}$, we define the second fundamental form $A$ by
\begin{equation}\label{eq:def2ff}
    \widetilde{\n}_X Y = \widehat{\n}_X Y +\e A(X,Y) \nu
\end{equation}
where $\e = \ip{\nu}{\nu}$ and $\nu$ is a local unit normal to $\widehat{M}^n$. As a result, we have the identity
\[-\ip{\widetilde{\n}_Y \nu}{X} = \ip{\nu}{\widetilde{\n}_Y X}=A(X,Y)\ .\]
These conventions give rise to the Codazzi and Gauss equations given by
\[\ip{\widetilde{R}(X,Y)Z}{\nu} = \widehat{\n}_Y A(X,Z) - \widehat{\n}_X A(Y,Z)\]
and
\[\ip{\widetilde{R}(X,Y)Z}{V} = \ip{\widehat{R}(X,Y)Z}{V}+\e[A(Y,Z)A(X,V)-A(X,Z)A(V,Y)]\ .
\]

By differentiating the defining equation for $S^{n,1}_1(\rho)\subset \bb{R}^{n+1,1}$ at $p\in S^{n,1}_\rho$ for outward normal $\ov{\nu} = \rho^{-1} p$ we have
\[\ov{A}_{\alpha\beta} = -\rho^{-1}\ov{g}_{\alpha\beta}\]
and so the Gauss equation yields
\[\ov{R}_{\a\b\g\d} = \rho^{-2}\left(\ov{g}_{\a\g}\ov{g}_{\b\d}-\ov{g}_{\b\g}\ov{g}_{\a\d}\right)\ .\]
Applying \eqref{eq:def2ff} to $M\subset S^{n,1}_\rho\subset \bb{R}^{n+1,1}$ we have the Gauss formula
  \begin{equation}
  \label{eq:6}
\ov{\ov{\n}}_i X_j = \nabla_{i}X_{j} - A_{ij}\nu - \rho^{-1}g_{ij} \rho^{-1}X.
\end{equation}
the Codazzi equation
\begin{equation}
  \label{eq:Codazzi}
  \nabla_{k}A_{ij} = \nabla_{i}A_{kj}.
\end{equation}
and the Gauss equation
\begin{equation}
    \label{eq:Gauss_eq}
R_{ijkl} =A_{il}A_{jk}- A_{ik}A_{jl} +\rho^{-2}\left(g_{ik}g_{jl}-g_{il}g_{jk}\right)
\end{equation}
Note that the Ricci and the scalar curvature tensors in orthonormal frame are given respectively by
\begin{equation}
  \label{eq:ricciCT}
  R_{ik} = g^{jl}R_{ijkl} =  A_{jk}A^{j}_{i}- H A_{ik}  + \rho^{-2}(n -1)g_{ik},
\end{equation}
and
\begin{equation}
  \label{eq:ScalarCT}
  R = g^{ik}R_{ik} =  |A|^2 - H^2  + \rho^{-2}n(n-1).
\end{equation}
 
We also recall the well known formula for derivative interchange 
\begin{equation}
    \nabla_{k}\nabla_{l} A_{ij} - \nabla_{l}\nabla_{k} A_{ij} = R_{kljr}A_{i}^r + R_{klir}A_{j}^r. \label{eq:Interchange}
\end{equation}

\subsection{Graphical hypersurfaces in de Sitter space}
     We denote the standard Euclidean round unit sphere by  $(S^n,\sigma)$, which we will think of as being the standard embedding in $\on{span}\{E_1, \ldots, E_{n+1}\}\subset \bb{R}^{n+1,1}$.
     Given $r\in \bb{R}$ and $\xi = \xi(x^1,\ldots,x^n) \in S^n$,  we can re-write (\ref{eq:2})
     in polar coordinates and identify $S^{n,1}_\rho$ with the points 
     \begin{equation}
     Y(r, \xi) = \rho \sinh r E_0+\rho\cosh(r)\xi \in \bb{R}^{n+1,1}. 
     \end{equation}
Let $u:\mathbb{S}^n\to\mathbb{R}$ be a function and consider $M=\mbox{graph}(u)$, parametrised by
$X:\mathbb{S}^n\to S^{n,1}_\rho$ given by $X(\xi)=Y(\sinh(u(\xi)),\xi)$. We will use $D$ to denote the standard covariant derivative for the metric $\sigma$ on $\mathbb{S}^n$, and throughout we will require the embedding to be spacelike.

Note here that we are choosing a particular isometrically embedded $S^n$ as our ``equator'' given by the zero graph $u=0$ based on our choice of $E_0$. Due to the invariance of $S^{n,1}_\rho$ under isometries of $\bb{R}_{1}^{n+2}$ and the noncompactness of possible choices of $E_0$, the choice of graph direction will be important for us later.

\noindent
Writing the unit timelike vector $E_r := \rho^{-1}\frac{\partial Y}{\partial r}$ we have that
\begin{equation}
    g_{ij} = \rho^2(\cosh^2(u) \sigma_{ij} - u_iu_j)
\end{equation}
the normal vector $\nu$ can be written as
\begin{equation*}
  \nu =-\frac{1}{\sqrt{\cosh^2 u-|Du|^2}}(\cosh(u)E_r+Du).
\end{equation*}
where as usual $Du = u_i\sigma^{ij}\xi_i$. 

We define the tilt of $M$ as
\begin{equation}
\label{eq:tilt}
    \tau: =\langle\nu, E_0\rangle,
\end{equation}
and the height function by
\begin{equation}
\label{eq:height}
\eta := - \langle X , E_0 \rangle,  
\end{equation}
and so in terms of the graph function $u$ we have
\begin{equation}
    \eta = \rho \sinh u, \qquad \tau = \frac{\cosh^2 (u)}{\sqrt{\cosh^2(u) - |Du|^2}}\ .
\end{equation}
$M$ is spacelike at a point if $g_{ij}$ is positive definite, which is equivalent to $\tau$ being finite at that point. We note that with an upper bound on $\tau$ we have complete control of the eigenvalues of $g_{ij}$. A standard calculation yields that
 \begin{equation}
  \label{eq:Secondff}
  A_{ij} = \rho\cosh^{-1}(u)\tau
  \left(D^2_{ij}u - 2\tanh(u)D_{i}u D_{j}u + \sinh(u)\cosh(u)\sigma_{ij}\right).
\end{equation}

\subsection{The PDE} 
We consider the isometric embedding as a prescribed curvature problem. These will be written in terms of symmetric polynomials of principal curvatures. We define the k-symmetric polynomial
\[P_k(\l_1, \ldots, \l_n) = \sum_{1\leq i_1<i_2<\ldots<i_k\leq n} \l_{i_i}\l_{i_2}\ldots \l_{i_k}\ .\]
We will be particularly interested in $P_2(\vec{\lambda})$ where $\vec{\lambda} = (\lambda_1,...,\lambda_n)$ are the eigenvalues of $A$. This may also be written as $P_2 = \frac 1 2(H^2-|A|^2)$, where here we using the definition that $H=P_1(\vec{\l})$. We recall that curvature cone $\Gamma_2$ is the connected component of $\{x\in \bb{R}^n| P_2(x)>0\}$ which contains the positive cone. A hypersurface will be said to be \emph{admissible} if at every point its principal curvatures satisfy  $\vec{\lambda}\in \Gamma_2$.
We also define 
\[ P_{k,i}(\vec{\lambda}) := P_{k}(\lambda_1,\ldots,\lambda_{i-1},0,\lambda_{i+1},\ldots, \lambda_n) 
\]
\noindent
By \eqref{eq:ScalarCT} we see that we have the equation
\begin{equation}
\label{eq:prescription}
2F(\vec{\l}):=2P_2(\vec{\l}) = \rho^{-2}n(n-1) - R = : \psi_\rho,
\end{equation}
where we see that $\psi_\rho$ depends only on the prescribed metric and $\rho$. This is a prescribed symmetric curvature functional equation, and the condition for this to be admissible for convex surfaces is $\rho^{-2}n(n-1)-R>0$. This is an upper bound requirement on the total curvature - the opposite to Euclidean space. This comes from the sign on the normal in the Gauss equations. 

We note here that $f(\vec{\lambda}):=\sqrt{F(\vec{\lambda})}$ is a concave and we may calculate that
\[F^{ij} = Hg^{ij}-A^{ij}\ .\]

\section{Preliminary lemmata}
\begin{lem}
We have
\begin{flalign}
 \n_i \eta & = -\ip{E_0}{X_i}\\
\n_i \tau &= -A_i^j\ip{X_j}{E_0} = A_i^j\n_j\eta \label{eq:ntau} \\
& | \nabla \eta|^2 +\rho^{-2}\eta^2+1= \tau^2 \label{eq:neta2}
\end{flalign}
\end{lem}
\begin{proof}
The calculations of first two derivatives is immediate from the definitions and our above sign conventions.

For \eqref{eq:neta2}, the ambient vector $E_0$ may be decomposed in the following coordinates 
\begin{equation}
\label{eq:e1}
  \begin{split}
    E_{0} &= \nabla\eta - \tau \nu - \rho^{-2}\eta X.
    \end{split}
  \end{equation}
and so we have
\begin{equation}
 -1= | \nabla \eta|^2- \tau^2+\rho^{-2}\eta^2\ .
  \end{equation}
 
\end{proof}
\begin{lem}
\label{lem:ddeta_tau_A}
We have the following second order relations.
\begin{flalign}
\label{eq:etarho}\n^2_{ij} \eta &=A_{ij} \tau - \rho^{-1}g_{ij}\eta\\
\label{eq:taurho}\n^2_{ij} \tau &=\n^k \eta\n_{k}A_{ij} + \tau A^2_{ij} - \rho^{-1}A_{ij} \eta\\
\label{Simonsrho}\Delta A_{ij} - \n^2_{ij} H&= |A|^2A_{ij} - HA^2_{ij} +\rho^{-2}[nA_{ij} - Hg_{ij}]
\end{flalign}
\end{lem}
\begin{proof}

We calculate
\begin{flalign*}
\n^{2}_{ij} \eta & =  - \ip{E_0}{D^2_{ij} X - \n_{X_i} X_j}\\
 & =  - \ip{E_0}{\ov{A}_{ij}X-A_{ij} \nu}\\
 & =\ov{A}_{ij} \eta +A_{ij} \tau\\
 &=A_{ij} \tau - \rho^{-1}g_{ij}\eta
\end{flalign*}
Next
\begin{flalign*}
\n^2_{ij} \tau & = \n_i A_j^k \n_k \eta + A_j^k \n_{ik}\eta \\
&=\n^k \eta\n_{k}A_{ij} + \tau A^2_{ij} - \rho^{-1}A_{ij} \eta
\end{flalign*}
\noindent
Finally we demonstrate the Simons-like identity. Using the Codazzi\eqref{eq:Codazzi}, derivative interchange \eqref{eq:Interchange} and Gauss equations \eqref{eq:Gauss_eq} we have
\begin{flalign*}
\n^{\ k}\n_k A_{ij} - \n_i\n_j A_k^k&=\n_k\n_i A_{j}^k - \n_i\n_k A_j^k\\
&=R_{kils}g^{kl}A^s_j+R_{kijs}A^{ks}\\
&=A_i^kA_{ks}A^s_j - A_k^kA_{is}A^s_j + A_{ij}A_{ks}A^{ks} - A_{kj}A^{ks}A_{is}\\
&\qquad+\rho^{-2}(n-1)A_{ij}+\rho^{-2}A_{ij} - \rho^{-2}Hg_{ij}\\
&= |A|^2A_{ij} - HA^2_{ij} +\rho^{-2}nA_{ij} - \rho^{-2}Hg_{ij}\ .
\end{flalign*}
\end{proof}
We now briefly justify that for general isometric embedding problems, we need to scale de Sitter space to get reasonable estimates. We briefly demonstrate that we may not embed an arbitrarily small sphere into $S^{n,1}_1$.
\begin{lem}\label{lem:Nonexist}
Let $S^n_r$ be the standard round Euclidean sphere of radius $r\in (0,1)$. We have that
\begin{enumerate}[a)]
    \item For $n\geq 3$ there does not exist an isometric embedding $X:S_r^n \to S^{n,1}_1$.
    \item For $n=2$ there does not exist an admissible isometric embedding $X:S_r^2 \to S^{2,1}_1$.
\end{enumerate}
\end{lem}
\begin{proof}
This follows from the sectional curvature of the sphere and the Gauss equation. Suppose that such an embedding exists and take coordinates in the principal directions of the submanifold. As the embedding is isometric, \eqref{eq:Gauss_eq} implies for any $i\neq j$
\[\frac{1}{r^2} = R_{ijij} =- \l_i\l_j+1\]
so
\[0<\frac 1 {r^2} -1 =- \l_i\l_j\ ,\]
and so for any pair $i\neq j$, $\l_i$ and $\l_j$ must have opposite signs. In dimension 3 this is not possible. In dimension 2 this means that $P_2(\vec{\l})<0$ and so the embedding is nowhere admissible. 
\end{proof}
\section{Curvature in de Sitter space}
\noindent
We now prove an initial estimate for the mean curvature.
\begin{lem}
\label{lm1}
Suppose that we have a convex isometric immersion into de Sitter space with metric satisfying $\rho^{-2}-n^{-1}\psi_{\rho} > 0$ and $\psi_\rho >0$. Then

\begin{equation}
H^2 \leq \frac{(n-1)^{-1}}{\rho^{-2} - n^{-1}\psi_\rho}\left[\frac 1 2\Delta R-\psi_{\rho}^2 + n\rho^{-2}\psi_\rho\right].
\end{equation}

\end{lem}
\begin{proof}
At a maximum of $H$, we have $\n H=0$, $\n^2 H\leq 0$. As a result, since $Hg^{ij} - A^{ij}\geq 0$, since $\psi_\rho >0$ for small $\rho$ and $\partial P_{2}/\partial \lambda_j \geq 0$, (see \cite{CNS3}), 
\[H\Delta H - A^{ij} \n_{i}\n{j}H \leq 0\ .\] 
Using the identity \eqref{Simonsrho} from Lemma \ref{lem:ddeta_tau_A} , we compute as follows
\begin{flalign*}
\frac 1 2\Delta R &= A^{ij} \Delta A_{ij} - H\Delta H+|\n A|^2 - |\n H|^2\\
&\geq  A^{ij} (\Delta A_{ij} - \n_{i}\n_{j} H)\\
&= |A|^4-H\tr A^3+\rho^{-2}\left(n|A|^2  - H^2\right) 
\end{flalign*}
From equations \eqref{eq:prescription} we have that $|A|^2 =  H^2 - \psi_\rho$ and so applying this twice we get
\begin{flalign*}
\frac 1 2\Delta R 
&\geq H(H|A|^2-\tr A^3)-\psi_\rho H^2+\rho^{-2}(n-1)H^2 +\psi_\rho^2-n\rho^{-2}\psi_\rho \ .
\end{flalign*}
Following Urbas \cite[equation 3.8]{Urbas:cespacelike}, we note that using standard relations on symmetric polynomials
\[
H|A|^2 - \tr A^3 = \sum_{i=1}^n P_{1,i}(\vec{\l})\l_i^2= P_1(\vec{\l})P_2(\vec{\l})-3P_3(\vec{\l})\]
and so
\[H|A|^2 - \tr A^3 \geq P_1(\vec{\l})P_2(\vec{\l})-\frac{2(n-2)P_2^2(\vec{\l})}{(n-1)P_1(\vec{\l})}
\geq \frac{2}{n}P_1(\vec{\l})P_2(\vec{\l}) = \frac{1}{n}H\psi_\rho
\]
where the first inequality uses the Newton inequalities (which apply regardless of curvature cone) and the second inequality follows from the Maclaurin inequality in $\Gamma_2$.

Therefore we obtain that 
\begin{equation}
\begin{split}
    \frac{1}{2}\Delta R & \geq (n^{-1} - 1) \psi_{\rho} H^2 + \psi_{\rho}^2 + \rho^{-2}(n-1) H^2+\psi_\rho^2-n\rho^{-2}\psi_\rho\\
    & = (n-1) (\rho^{-2} - n^{-1} \psi_\rho) H^2 + \psi_\rho^2 - n\rho^{-2}\psi_\rho,
    \end{split}
\end{equation}    
from which the claimed inequality follows.
\end{proof}

\section{Proof of Theorem 1}
Recall that we have 
\[ \psi_\rho = \rho^{-2}n(n-1) - R. \]
In order to apply Lemma~\ref{lm1}, we require 
$ 0 < \psi_\rho < n\rho^{-2}$, equivalently
\begin{equation}
    \rho^{-2}n(n-2) < R < \rho^{-2}n(n-1).
\end{equation}
Note that we can write $H^2 = |A|^2 + 2 P_2$, and since we are assuming 2-convexity, the estimate for $|A|^2$ will follow from the estimate of $H^2$.

If $n = 2$ then the left hand side becomes zero and, as $R$ has a finite maximum. We may
see that if $\rho$ is small enough we may obtain a curvature estimate depending only on $R$
and it’s derivative.

\section{Proof of Theorem 2}
\noindent
Our first observation is that we need to choose a coordinate system that our estimates make reference to. This is because the Lorentz group $SO(1,n+1)$ of isometries fixing the origin of $\bb{R}^{n+1,1}$ is 
noncompact, and given any spacelike surface, and a given point on that surface, a rotated surface will have an arbitrarily large tilt at the corresponding point. As a result, it is clear that we must always choose a sensible graph direction or equivalently, sensible choices of base $E_0, \ldots, E_{n+1}$. However, the curvature estimates from Theorem \ref{thm:CurvEsts} hold regardless of graphical representation.
\noindent
Note from \eqref{eq:e1} that the tangential projection to the surface of the vector $E_0$, which we denote by $E^\top_0$, coincides with $\n \eta$, hence we have by equation \eqref{eq:neta2} we know
\begin{equation}
     \tau^2 =1 + \rho^{-2}\eta^2+|\n \eta|^2. \label{eq:tauidentity}
\end{equation}

Also from equation \eqref{eq:ntau}
we have that
\[|\n \tau|^2 \leq |A|^2|E_0^\top|^2\leq \tau^2 |A|^2\]
or, using Theorem \ref{thm:CurvEsts},
\[|\n \log \tau| \leq |A|\leq C_H\ .\]
We now pick our coordinate system: By choosing any point $p\in S^n$, we may choose a basis such that $\nu(p)$ is in direction $E_0$ which implies that $\tau(p)=1$. In particular, due to \eqref{eq:neta2} we also now know that $\eta(p)=0$. Therefore integrating from $p$ we have that 
\[\max\tau=\frac{\max\tau}{\tau(p)} \leq e^{C_H\operatorname{diam}(M)}=:C_\tau\ . \]
As $\operatorname{diam}(M)$ is intrinsic and $C$ depends only on $g$ and $\tau$ is bounded. By equation \eqref{eq:neta2} we also have that
\[|\eta|\leq C_\tau\ .\]
These estimates together supply the required $C^1$ bound as
\[\tau = \frac{\sqrt{1+\eta^2}}{\sqrt{1-\frac{|\tilde{\n} u|^2}{\eta^2+1}}}\ .\]
Using the graphical form of $A_{ij}$ we therefore know that $\eta$, and therefore $u$ are bounded in $C^2$.
\noindent
Our estimates now follow the familiar path from PDE theory. We note that, by setting $f=\sqrt{F}=\sqrt{P_2}$ we have that 
\[f(\lambda) = \sqrt{\frac{1}{2}(\rho^{-2}n(n-1)-R)}\]
which (by admissibility of the solution and the well-known convexity of $f$) is an elliptic PDE in $u$. Furthermore, due to our $C^2$ estimates the second fundamental form stays in a compact subset of the $\Gamma_2$ cone, and so this equation is uniformly elliptic. We may therefore apply Evans \cite{Evans} and Krylov-Safonov  \cite{KrylovSofonov} to see that $|u|_{C^{2+\alpha}}<C$ for some $\alpha\in(0,1)$. As our graph is now bounded in $C^{2+\alpha}$, $R$ is bounded in $C^\alpha$, and we may repeatedly apply elliptic H\"older estimates to get that $|u|_{C^{k+\alpha}}\leq C_{k+\alpha}$
for all $k\in \mathbb{N}$.

\end{document}